\documentclass{amsart}
\usepackage{amsmath,amsfonts,amsthm}
\usepackage{amssymb,latexsym}
\usepackage{graphics}
\usepackage[colorlinks]{hyperref}

\usepackage[all]{xypic}

\theoremstyle{plain}
\theoremstyle{definition}
\newtheorem{theorem}{Theorem}[section]
\newtheorem{lemma}[theorem]{Lemma}

\newtheorem{counterexample}[theorem]{Counterexample}

\theoremstyle{remark}

\numberwithin{equation}{section}

\title{Possible heights of graph transformation groups}
\author[F. Ayatollah Zadeh Shirazi, A. Hosseini, Z. Nili Ahmadabadi]{Fatemah Ayatollah Zadeh Shirazi, \\ Arezoo Hosseini, Zahra Nili Ahmadabadi}
\begin{document}
\begin{abstract}
In the following text we prove that for all finite $p\geq0$ there exists a topological graph $X$
such that $\{p,p+1,p+2,\ldots\}\cup\{+\infty\}$ is the collection of all possible heights
for transformation groups with phase space $X$. Moreover for all topological graph $X$ with
$p$ as height of transformation group $(Homeo(X),X)$,
$\{p,p+1,p+2,\ldots\}\cup\{+\infty\}$ again is the collection of all possible heights
for transformation groups with phase space $X$.
\end{abstract}
\maketitle
\noindent {\small {\bf 2010 Mathematics Subject Classification:}  54H20 \\
{\bf Keywords:}} Dynamical system, Height, Topological graph, Transformation group.
\section{Preliminaries}
\noindent By a transformation group  $(G,X,\rho)$ or simply $(G,X)$ we mean
a compact Hausdorff topological space $X$, discrete topological group $G$ with identity $e$,
and continuous map $\mathop{\rho:G\times X\to X}\limits_{\:\:\:\:\:\:\:\:\:\:\:(g,x)\mapsto gx}$ such that for all
$x\in X$ and $s,t\in G$ we have $ex=x$ and $s(tx)=(st)x$. In transformation group
$(G,X)$ we say the nonempty subset $Y$ of $X$ is invariant if $GY:=\{gy:g\in G,y\in Y\}\subseteq Y$~\cite{ellis, ellis2}.
For closed invariant subset $Y$ of $(G,X)$ we denote height of $Y$ (or $(G,Y)$)
by $h(G,Y)$, where $h(G,Y):=\sup\{n\geq0:$ there exist
distinct closed invariant subsets $Y_0\subset Y_1\subset\cdots\subset Y_n=Y$ of $(G,X)\}$
which has been introduced for the first time in~\cite{trans} (for more details see~\cite{ind}).
\\
By a dynamical system $(X,f)$ we mean a topological space $X$ and homeomorphism $f:X\to X$.
There exists a one to one correspondence between the collection of all dynamical systems $(X,f)$
with phase space $X$ and transformation groups $({\mathbb Z},X,\rho)$ (with $nx(=\rho(n,x))=f^n(x)$); more
simply one may consider transformation group $(\{f^n:n\in{\mathbb Z}\},X)$ instead of dynamical system
$(X,f)$.
\\
For topological space $A$ suppose $Homeo(A)$ denotes the collection of all homeomorphisms $f:A\to A$ and $P_h(A)=\{h(G,A):(G,A)$ is a transformation group$\}(=\{h(G,A):G$ is a subgroup of $Homeo(A)$ under the operation of composition of maps$\}$). For homeomorphism $f:A\to A$ let $<f>:=\{f^n:n\in{\mathbb Z}\}$.
\\
By a topological graph we mean a compact connected metric space $X$ which is finite union of arcs such that
every two arcs intersect each other in at most their end points (where each topological space homeomorph
with [0,1] is an arc and points in correspondence with 0,1 are end points) \cite{nadler} (moreover one may see \cite{111, 222}
as examples of papers on dynamical and topological properties of topological graphs).
\section{Unit interval}
\noindent In this section we prove that the collection of all possible heights for unit interval transformation group $(G,[0,1])$ is $\{1,2,3,\ldots\}\cup\{+\infty\}$.
\begin{lemma}\label{salam10}
In unit interval transformation group $(G,[0,1])$ we have
$h(G, [0,1])\geq1$ moreover $h(Homeo([0,1]),[0,1])=1$.
\end{lemma}
\begin{proof}
For each homeomorphism $f:[0,1]\to[0,1]$ we have $f\{0,1\}=\{0,1\}$, thus in transformatin group $(G,[0,1])$ we have $G\{0,1\}=\{0,1\}$ and $\{0,1\}$ is a closed proper invariant subset of $[0,1]$, hence $h(G, [0,1])\geq1$. We have:
\[Homeo([0,1])x=\left\{\begin{array}{lc}(0,1) & x\in(0,1)\:, \\ \{0,1\} & x=0,1 \:. \end{array}\right.\]
So $\{0,1\}\subset[0,1]$ is the only chain of closed invariant subsets of $(Homeo([0,1]),[0,1])$ and $h(Homeo([0,1]),[0,1])=1$.
\end{proof}
\begin{lemma}\label{salam20}
For Hausdorff topological space $A$ we have $+\infty\in P_h(A)$ if and only if $A$ is infinite.
\end{lemma}
\begin{proof}
If $A$ is infinite, let $G=\{id_A\}$, then $h(G,A)=+\infty$
(since for each $x\in A$, $\{x\}$ is a closed invariant subset of $(G,A)$). Now suppose there exists $K$ with $h(K,A)=+\infty$, thus $\{\overline{Kx}:x\in A\}$
is infinite, in particular the collection of subsets of $A$ is infinite and $A$ is infinite.
\end{proof}
\begin{theorem}\label{salam30}
$P_h([0,1])=\{1,2,3,\ldots\}\cup\{+\infty\}$.
\end{theorem}
\begin{proof}
By Lemmas~\ref{salam10} and \ref{salam20} we have $+\infty\in P_h([0,1])\subseteq\{1,2,3,\ldots\}\cup\{+\infty\}$. For $n\geq1$ choose $0=x_0<x_1<\cdots<x_n=1$ and let $K=\{f\in Homeo([0,1]):\forall i\in\{0,\ldots,n\}\:f(x_i)=x_{n-i}\}$, then $G=\{f_1\cdots f_s:s\geq1,f_1,\ldots,f_s\in K\}$ is a subgroup of $Homeo([0,1])$ (since $K=K^{-1}$) such that
$\{f\in Homeo([0,1]):\forall i\in\{0,\ldots,n\}\:f(x_i)=x_i\}\subseteq G$
(since $K\neq\varnothing$ and for all $k\in K$, $h\in \{f\in Homeo([0,1]):\forall i\in\{0,\ldots,n\}\:f(x_i)=x_i\}$ we have $hk,k^{-1}\in K$ moreover $h=(hk)k^{-1}\in G$). Therefore we have:
\[Gx=\left\{\begin{array}{lc} \{x_i,x_{n-i}\} & x=x_i,i=0,\ldots,n\:,\\ (x_{i-1},x_i)\cup(x_{n-i},x_{n-(i-1)}) & x\in(x_{i-1},x_i), i=1,\ldots,n \:. \end{array} \right.\]
Thus $\{\overline{Gx}:x\in[0,1]\}=\{\{x_i,x_{n-i}\}:0\leq i\leq n\}\cup\{[x_{i-1},x_i]\cup[x_{n-i},x_{n-(i-1)}]:1\leq i\leq n\}$ has $n+1$ elements and $h(G,[0,1])=n$.
\end{proof}
\begin{theorem}
For $f\in Homeo([0,1])$ we have $h(<f>,[0,1])=+\infty$.
\end{theorem}
\begin{proof}
For $f\in Homeo([0,1])$ we have the following two cases:
\\
{\it Case 1}: $f(0)=0,f(1)=1$. In this case $f$ is increasing and the only
periodic points of $f$ are fix points. If $f$ has infinite fix points, then $\{\{x\}:x$ is a fix point of $f\}$ is an infinite set of closed invariant subsets of $(<f>,[0,1])$ thus $h(<f>,[0,1])=+\infty$. Otherwise the collection of periodic (fix) points of $f$ is finite. Choose $x_1\in[0,1]\setminus Per(f)$,
then $x_1<f(x_1)$ or $f(x_1)<x_1$ ($f(x_1)<x_1$ leads to $x_1<f^{-1}(x_1)$, note to $<f>=<f^{-1}>$), we may suppose $x_1<f(x_1)$, thus
$\cdots f^{-2}(x_1)<f^{-1}(x_1)<x_1<f(x_1)<f^2(x_1)<\cdots$, choose
$x_2\in(x_1,f(x_1))$. It's evident that $x_2\notin\overline{\{f^n(x_1):n\in\mathbb{Z}\}}$ moreover $\cdots<f^{-2}(x_2)<f^{-1}(x_2)<x_2<f(x_1)<f(x_2)<\cdots$, choose $x_3\in(x_2,f(x_1))$, then $x_3\notin \overline{\{f^n(x_1):n\in\mathbb{Z}\}}\cup
\overline{\{f^n(x_2):n\in\mathbb{Z}\}}$. By using this method and choose sequence $x_1<x_2<x_3<\cdots<f(x_1)$ we have $x_k\notin\bigcup\{\overline{\{f^n(x_i):n\in\mathbb{Z}\}}:i<k\}$ for all $k>1$. Hence $\{\overline{<f>x_n}:n\geq1\}$ is infinite and $h(<f>,[0,1])=+\infty$.
\\
{\it Case 2}: $f(0)=1,f(1)=0$. By Case 1, $h(<f^2>,[0,1])=+\infty$, thus
$\{\overline{<f^2>x}:x\in[0,1]\}$ is infinite. For all $x\in[0,1]$ we have
$\overline{<f>x}=\overline{<f^2>x}\cup \overline{<f^2>f(x)}$, thus
$\{\overline{<f>x}:x\in[0,1]\}$ is infinite too and $h(<f>,[0,1])=+\infty$.
\end{proof}
\section{Circle}
\noindent Consider unit circle ${\mathbb S}^1:=\{e^{2\pi i\theta}:\theta\in[0,1]\}$. In this section we prove that the collection of all possible heights for circle transformation group $(G,{\mathbb S}^1)$ is $\{0,1,2,3,\ldots\}\cup\{+\infty\}$.
\begin{theorem}\label{salam35}
$P_h({\mathbb S}^1)=\{0,1,2,3,\ldots\}\cup\{+\infty\}$.
\end{theorem}
\begin{proof}
$(Homeo({\mathbb S}^1),{\mathbb S}^1)$ is minimal and $h(Homeo({\mathbb S}^1),{\mathbb S}^1)=0$ as a matter of fact for $\alpha\in{\mathbb R}$ if $\varphi_\alpha(z)=ze^{2\pi i\alpha}$ ($z\in{\mathbb S}^1$), then:
\[h(<\varphi_\alpha>,{\mathbb S}^1)=\left\{\begin{array}{lc} 0 & \alpha\in{\mathbb R}\setminus{\mathbb Q}\:, \\ +\infty & \alpha\in{\mathbb Q}\:.\end{array}\right.\]
Using the same notations as in the proof of Theorem~\ref{salam30} for $n\geq1$ there exists subgroup $G$ of $Homeo([0,1])$ with $G0=G1=\{0,1\}$ and $|\{\overline{Gx}:x\in[0,1]\}|=n+1$. For $f\in G$ define $\psi_f:{\mathbb S}^1\to{\mathbb S}^1$ with $\psi_f(e^{2\pi i\theta})=e^{2\pi i f(\theta)}$ ($\theta\in[0,1]$). Then $H=\{\psi_f:f\in G\}$ is a subgroup of $Homeo({\mathbb S}^1)$ and $\{\overline{Hz}:z\in{\mathbb S}^1\}=\{\{e^{2\pi iw}:w\in\overline{Gx}\}:x\in [0,1]\}$ moreover $\eta:\{\overline{Gx}:x\in[0,1]\}\to\{\overline{Hz}:z\in{\mathbb S}^1\}$
with $\eta(\overline{Gx})=\{e^{2\pi iw}:w\in\overline{Gx}\}$ is bijective. Thus $|\{\overline{Hz}:z\in{\mathbb S}^1\}|=|\{\overline{Gx}:x\in[0,1]\}|$ and $h(H,{\mathbb S}^1)=h(G,[0,1])=n$.
\end{proof}
\section{Star}
\noindent For $n\geq3$ suppose $\ast_n$ is the star graph with $n+1$ vertices $\{0,1,\ldots,n\}$ and $n$ edges $\{[0,1],[0,2],\ldots,[0,n]\}$ (the following graph):
\vspace{3mm}
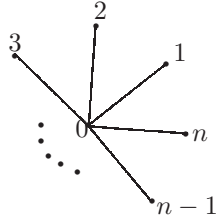
\begin{center}
\unitlength .4mm 
\linethickness{0.4pt}
\ifx\plotpoint\undefined\newsavebox{\plotpoint}\fi 
\begin{picture}(83.426,85.938)(0,0)
\multiput(37.148,38.921)(.0674224344,-.080548926){419}{\line(0,-1){.080548926}}
\multiput(37.398,38.421)(.89285714,-.06632653){49}{\line(1,0){.89285714}}
\multiput(37.398,38.921)(.0853960396,.067450495){404}{\line(1,0){.0853960396}}
\put(71.898,66.171){\line(0,1){0}}
\put(71.898,66.171){\line(0,1){0}}
\multiput(37.648,38.921)(.06666667,.98888889){45}{\line(0,1){.98888889}}
\put(40.648,83.421){\line(0,1){0}}
\put(40.648,83.421){\line(0,1){0}}
\put(40.648,83.421){\line(0,1){0}}
\put(40.648,83.421){\line(0,1){0}}
\multiput(37.148,38.921)(-.0690578158,.0674518201){467}{\line(-1,0){.0690578158}}
\put(71.898,65.921){\circle*{2.5}}
\put(40.898,82.921){\circle*{2.5}}
\put(4.949,69.677){\circle*{2.5}}
\put(19.898,25.171){\circle*{2.5}}
\put(16.648,32.171){\circle*{2.5}}
\put(16.648,38.921){\circle*{2.5}}
\put(25.148,21.921){\circle*{2.5}}
\put(32.648,18.171){\circle*{2.5}}
\put(37.398,38.921){\circle*{2.5}}
\put(65.648,5.171){\circle*{2.5}}
\put(80.648,35.171){\circle*{2.5}}
\put(67.648,-.329){$n-1$}
\put(83.426,31.881){$n$}
\put(75.114,67.421){$1$}
\put(39.898,85.938){$2$}
\put(2.04,72.029){$3$}
\put(31.648,33.921){$0$}
\end{picture}
\end{center}
In this section we prove that $P_h(\ast_n)=\{2,3,\ldots\}\cup\{+\infty\}$.
\begin{lemma}\label{salam40}
For $n\geq3$, in graph transformation group $(G, \ast_n)$ we have
\linebreak
$h(G, \ast_n)\geq2$, moreover $h(Homeo(\ast_n),\ast_n)=2$.
\end{lemma}
\begin{proof}
Since $\{0\}\subset\{0,1,\ldots,n\}\subset\ast_n$ is a chain of closed invariant subsets of $(G, \ast_n)$, we have $h(G, \ast_n)\geq2$. Moreover:
\[\overline{Homeo(\ast_n)x}=\left\{\begin{array}{lc} \{0\} & x=0\:, \\
\{1,\ldots,n\} & x=1,\ldots,n\:, \\ \ast_n & x\in\ast_n\setminus\{0,1,\ldots,n\}\:.
\end{array}\right.\]
Thus $\{\overline{Homeo(\ast_n)x}:x\in\ast_n\}$ has $3$ elements and $h(Homeo(\ast_n),\ast_n)=2$.
\end{proof}
\begin{theorem}\label{salam50}
For $n\geq3$ we have
$P_h(\ast_n)=\{2,3,\ldots\}\cup\{+\infty\}$.
\end{theorem}
\begin{proof}
By Lemmas~\ref{salam40} and~\ref{salam20},
 $+\infty,2\in P_h(\ast_n)\subseteq\{2,3,\ldots\}\cup\{+\infty\}$.
 For $p\geq1$ and $i\in\{1,\ldots,n\}$ choose $x_0=0<x_1<\cdots<x_p=1$ in edge $[0,1]$
 and let $G=\{f\in Homeo(\ast_n):\forall i\in\{1,\ldots,p\}\:f(x_i)=x_i\}$. Then:
\[\overline{Gx}=\left\{\begin{array}{lc} \{x\} & x=x_0,\ldots,x_p\:,\\
\left[x_{i-1},x_i\right] & x_{i-1}<x<x_i , i=1,\ldots,p\:, \\
\{2,\dots,n\} & x=2,\ldots,n\: , \\
\ast_n\setminus(0,1] & x\in\ast_n\setminus([0,1]\cup\{2,\dots,n\})  \: . \end{array}\right.\]
And $\{\overline{Gx}:x\in\ast_n\}$ has $2p+3$ thus $h(G,\ast_n)=2p+2$.
\\
Choose $\{y_k\}_{k\in{\mathbb Z}}$ in $[x_0,x_1]$ with $x_0<\cdots<y_{-2}<y_{-1}<y_0<y_1<y_2\cdots<x_1$ and
${\displaystyle\lim_{k\to+\infty}y_k}=x_1$, ${\displaystyle\lim_{k\to+\infty}y_{-k}}=x_0$. Let $H=\{f\in G:f(\{y_k:k\in{\mathbb Z}\})=\{y_k:k\in{\mathbb Z}\}\}$, then:
\[\overline{Hx}=\left\{\begin{array}{lc} \{x\} & x=x_0,\ldots,x_p\:,\\
\{y_k:k\in{\mathbb Z}\}\cup\{x_0,x_1\} & x\in \{y_k:k\in{\mathbb Z}\}\: , \\
\left[x_{i-1},x_i\right] & x_{i-1}<x<x_i , i=1,\ldots,p\:, \\
\{2,\dots,n\} & x=2,\ldots,n\: , \\
\ast_n\setminus(0,1] & x\in\ast_n\setminus([0,1]\cup\{2,\dots,n\})  \: . \end{array}\right.\]
And $\{\overline{Hx}:x\in\ast_n\}$ has $2p+4$ thus $h(H,\ast_n)=2p+3$.
\\
Choose sequences $\{w^i_k\}_{k\in{\mathbb Z}}$ in $[0,i]$ with $0<\cdots<w^i_{-2}<w^i_{-1}<w^i_0<w^i_1<w^i_2\cdots<i$ and
${\displaystyle\lim_{k\to+\infty}w^i_k}=i$, ${\displaystyle\lim_{k\to+\infty}y_{-k}}=0$. Let $M=\{f\in Homeo(\ast_n):\forall i\in\{0,\ldots,n\}\:f(\{w^i_k:k\in{\mathbb Z}\})=\{w^{i+1}_k:k\in{\mathbb Z}\}\}$ (where $i+1$ has been considered module $n$, so for all $f\in M$ we have $f(\{w^k_k:k\in{\mathbb Z}\})=\{w^0_k:k\in{\mathbb Z}\}$), then:
\[\overline{Mx}=\left\{\begin{array}{l} \{0\} \: \: \: \: \: x=0\:,\\
\{1,\ldots,n\} \: \: \: \: \: x=1,\ldots,n\: , \\
\{w^i_k:k\in{\mathbb Z},1\leq i\leq n\}\cup\{0,1,\ldots,n\} \: \: \: \: \: x\in \{w^i_k:k\in{\mathbb Z},1\leq i\leq n\}\: , \\
\ast_n \: \: \: \: \: x\in\ast_n\setminus(\{w^i_k:k\in{\mathbb Z},1\leq i\leq n\}\cup\{0,\ldots,n\})\:.
\end{array}\right.\]
And $\{\overline{Mx}:x\in\ast_n\}$ has $4$ elements thus $h(M,\ast_n)=3$, which completes the proof.
\end{proof}
\section{An arising question}
\noindent Using previous sections we have
\begin{itemize}
\item[] $P_h({\mathbb S}^1)=\{0,1,2,\cdots\}\cup\{+\infty\}$,
\item[] $P_h([0,1])=\{1,2,3,\cdots\}\cup\{+\infty\}$,
\item[] $P_h(\ast_n)=\{2,3,4\cdots\}\cup\{+\infty\}$ ($n\geq3$).
\end{itemize}
Now this question arises: For $p\geq0$, is there any graph like $\mathcal{G}$ with
\linebreak
$P_h(\mathcal{G})=\{p,p+1,p+2,\cdots\}\cup\{+\infty\}$?
\\
In this section we prove that the above question has positive answer.
\begin{lemma}\label{salam60}
In transformation group $(G,X)$ if $H$ is a subgroup of $G$, then $h(G,X)\leq h(H,X)$,
in particular $p_h(X)\subseteq \{i\geq0:i\geq h(Homeo(X),X)\}\cup\{+\infty\}$.
\end{lemma}
\begin{proof}
Use the fact that for $x,y\in X$ if $\overline{Hx}=\overline{Hy}$, then $\overline{Gx}=\overline{Gy}$,
thus $|\{\overline{Gw}:w\in X\}|\leq|\{\overline{Hw}:w\in X\}|$.
\end{proof}
\begin{lemma}\label{salam70}
For all $p\geq0$ with $p\equiv2(\mod 4)$ there exists graph $X$ with $P_h(X)=\{p,p+1,\ldots\}\cup\{+\infty\}$.
\end{lemma}
\begin{proof}
For $n\geq3$ suppose $X_n$ is the following graph:
\vspace{3mm}
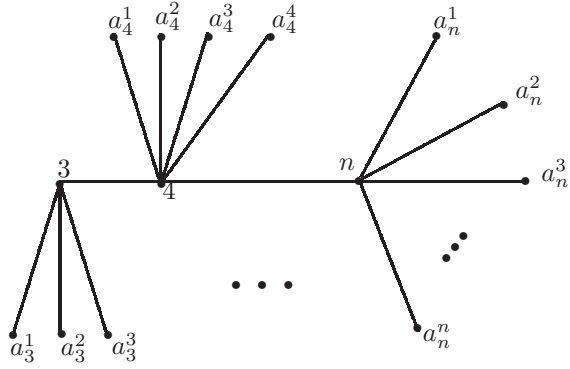
\begin{center}
\unitlength .4mm 
\linethickness{0.4pt}
\ifx\plotpoint\undefined\newsavebox{\plotpoint}\fi 
\begin{picture}(150.25,89.75)(0,0)
\thicklines
\put(49,47.25){\line(0,1){39.75}}
\put(22.25,47.25){\line(0,-1){39.75}}
\multiput(49,47)(.0672043,.21639785){186}{\line(0,1){.21639785}}
\multiput(49,46.75)(-.0672043,.21639785){186}{\line(0,1){.21639785}}
\multiput(22.25,47.5)(-.0672043,-.21639785){186}{\line(0,-1){.21639785}}
\multiput(22.25,47.25)(.0672043,-.21639785){186}{\line(0,-1){.21639785}}
\multiput(49.25,47.5)(.067375887,.092789598){423}{\line(0,1){.092789598}}
\put(36.5,86.25){\circle*{2}}
\put(122.25,86.5){\circle*{2}}
\put(140.25,68){\circle*{2}}
\put(146,48){\circle*{2}}
\put(117.25,8.75){\circle*{2}}
\put(101.75,48){\circle*{2}}
\put(69,20.25){\circle*{2}}
\put(76,20.25){\circle*{2}}
\put(83,20.25){\circle*{2}}
\put(129.5,33.25){\circle*{2}}
\put(127.25,30.25){\circle*{2}}
\put(125,27.25){\circle*{2}}
\put(49.25,86.25){\circle*{2}}
\put(61.75,86.25){\circle*{2}}
\put(78.25,86.25){\circle*{2}}
\put(49.25,47){\circle*{2}}
\put(22.25,47){\circle*{2}}
\put(35,7){\circle*{2}}
\put(22.75,7.25){\circle*{2}}
\put(9.75,7){\circle*{2}}
\multiput(102.25,49)(.067434211,.124177632){304}{\line(0,1){.124177632}}
\multiput(102,48.25)(.125,.067434211){304}{\line(1,0){.125}}
\multiput(101.75,48.25)(.06730769,-.17200855){234}{\line(0,-1){.17200855}}
\put(49.5,43){$4$}
\put(96.5,50.75){$n$}
\put(21.5,49){$3$}
\put(9,1.5){$a_3^1$}
\put(22.5,1.5){$a_3^2$}
\put(36,1.75){$a_3^3$}
\put(35,88.5){$a_4^1$}
\put(47.5,89.75){$a_4^2$}
\put(61.75,89.25){$a_4^3$}
\put(78.5,89){$a_4^4$}
\put(121.5,88.5){$a_n^1$}
\put(143.25,70){$a_n^2$}
\put(150.25,48){$a_n^3$}
\put(118.75,5.25){$a_n^n$}
\thinlines
\put(145.5,47.75){\line(-1,0){123.25}}
\end{picture}
\end{center}
then:
{\small
\[\overline{Homeo(X_n)x}=\left\{\begin{array}{ll} \{x\} & x=3,\ldots,n\:,\\
\{a_i^1,a_i^2,\ldots,a_i^i\} & x=a_i^j,i=3,\ldots,n,j=1,\ldots,i\:,\\
\bigcup\{\left[i,a_i^j\right]:j=1,\ldots,i\} & x\in\bigcup\{(i,a_i^j):j=1,\ldots,i\},i=3,\ldots,n\:,\\
\left[i,i+1\right] & x\in(i,i+1),i=3,\ldots,n-1\:.
\end{array}\right.\]}
So $\{\overline{Homeo(X_n)x}:x\in X_n\}$ has $4n-9$ elements, and $h(Homeo(X_n),x)=4n-10$.
Using Lemmas~\ref{salam20} and~\ref{salam60} we have $4n-10,+\infty\in P_h(X_n)\subseteq
\{4n-10,4n-9,4n-8,\ldots\}\cup\{+\infty\}$. Consider $t\geq2$. Since $\ast_3=[3,a_3^1]\cup[3,a_3^2]\cup[3,a^3_3]$ is a star graph with $3$ edges, by Theorem~\ref{salam50} there exists $K(\subseteq Homeo(\ast_3))$ with $h(K,[3,a_3^1]\cup[3,a_3^2]\cup[3,a^3_3])=t$. Moreover $K3=\{3\}$. Let $G=\{f\in Homeo(X_n):f\restriction_{[3,a_3^1]\cup[3,a_3^2]\cup[3,a^3_3]}\in K\}$, then $\{\overline{Gx}:x\in X_n\}=\{\overline{Kx}:x\in [3,a_3^1]\cup[3,a_3^2]\cup[3,a^3_3]\}\cup\{\overline{Homeo(X_n)x}:x\in X_n\setminus([3,a_3^1]\cup[3,a_3^2]\cup[3,a^3_3])\}$ has $4n-12+(t+1)$ elements, thus $h(G,X_n)=4n+t-12$. Hence $P_h(X_n)=\{4n-10,4n-9,4n-8,\ldots\}\cup\{+\infty\}$. Using $n\geq3$, leads us to the desired result.
\end{proof}
\begin{lemma}\label{salam80}
For all $p\geq0$ with $p\equiv3(\mod 4)$ there exists graph $X$ with $P_h(X)=\{p,p+1,\ldots\}\cup\{+\infty\}$.
\end{lemma}
\begin{proof}
For $n\geq3$ suppose $Y_n$ is the following graph:
\vspace{3mm}
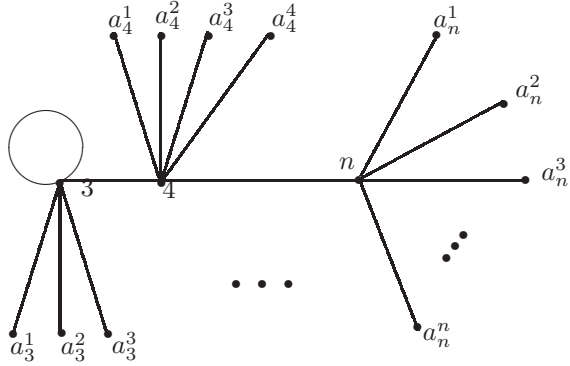
\begin{center}
\unitlength .4mm 
\linethickness{0.4pt}
\ifx\plotpoint\undefined\newsavebox{\plotpoint}\fi 
\begin{picture}(150.25,89.75)(0,0)
\thicklines
\put(49,47.25){\line(0,1){39.75}}
\put(22.25,47.25){\line(0,-1){39.75}}
\multiput(49,47)(.0672043,.21639785){186}{\line(0,1){.21639785}}
\multiput(49,46.75)(-.0672043,.21639785){186}{\line(0,1){.21639785}}
\multiput(22.25,47.5)(-.0672043,-.21639785){186}{\line(0,-1){.21639785}}
\multiput(22.25,47.25)(.0672043,-.21639785){186}{\line(0,-1){.21639785}}
\multiput(49.25,47.5)(.067375887,.092789598){423}{\line(0,1){.092789598}}
\put(36.5,86.25){\circle*{2}}
\put(122.25,86.5){\circle*{2}}
\put(140.25,68){\circle*{2}}
\put(146,48){\circle*{2}}
\put(117.25,8.75){\circle*{2}}
\put(101.75,48){\circle*{2}}
\put(69,20.25){\circle*{2}}
\put(76,20.25){\circle*{2}}
\put(83,20.25){\circle*{2}}
\put(129.5,33.25){\circle*{2}}
\put(127.25,30.25){\circle*{2}}
\put(125,27.25){\circle*{2}}
\put(49.25,86.25){\circle*{2}}
\put(61.75,86.25){\circle*{2}}
\put(78.25,86.25){\circle*{2}}
\put(49.25,47){\circle*{2}}
\put(22.25,47){\circle*{2}}
\put(35,7){\circle*{2}}
\put(22.75,7.25){\circle*{2}}
\put(9.75,7){\circle*{2}}
\multiput(102.25,49)(.067434211,.124177632){304}{\line(0,1){.124177632}}
\multiput(102,48.25)(.125,.067434211){304}{\line(1,0){.125}}
\multiput(101.75,48.25)(.06730769,-.17200855){234}{\line(0,-1){.17200855}}
\put(49.5,43){$4$}
\put(96.5,50.75){$n$}
\put(27.75,43.5){$3$}
\put(9,1.5){$a_3^1$}
\put(22.5,1.5){$a_3^2$}
\put(36,1.75){$a_3^3$}
\put(35,88.5){$a_4^1$}
\put(47.5,89.75){$a_4^2$}
\put(61.75,89.25){$a_4^3$}
\put(78.5,89){$a_4^4$}
\put(121.5,88.5){$a_n^1$}
\put(143.25,70){$a_n^2$}
\put(150.25,48){$a_n^3$}
\put(118.75,5.25){$a_n^n$}
\thinlines
\put(145.5,47.75){\line(-1,0){123.25}}
\put(18.5,56.5){\circle{19.81}}
\end{picture}
\end{center}
Using a similar method described in Lemma~\ref{salam70} we have $P_h(Y_n)=\{4n-9,4n-8,4n-7,\ldots\}\cup\{+\infty\}$, which leads to the desired result.
\end{proof}
\begin{lemma}\label{salam90}
For all $p\geq0$ with $p\equiv1(\mod 4)$ there exists graph $X$ with $P_h(X)=\{p,p+1,\ldots\}\cup\{+\infty\}$.
\end{lemma}
\begin{proof}
For $n\geq3$ suppose $Z_n$ is the following graph:
\vspace{3mm}
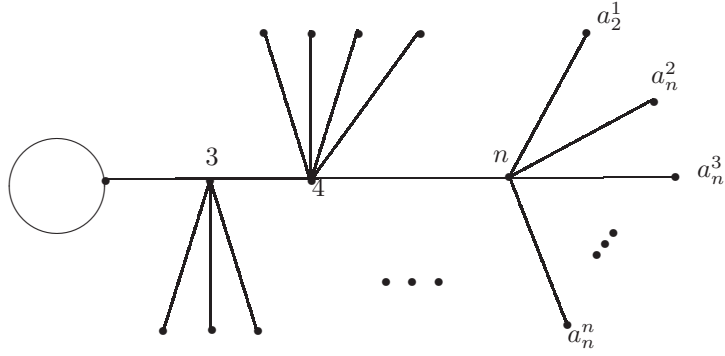
\begin{center}
\unitlength .4mm 
\linethickness{0.4pt}
\ifx\plotpoint\undefined\newsavebox{\plotpoint}\fi 
\begin{picture}(185.75,92)(0,0)
\thicklines
\put(83,50.25){\line(0,1){39.75}}
\put(56.25,50.25){\line(0,-1){39.75}}
\multiput(83,50)(.0672043,.21639785){186}{\line(0,1){.21639785}}
\multiput(83,49.75)(-.0672043,.21639785){186}{\line(0,1){.21639785}}
\multiput(56.25,50.5)(-.0672043,-.21639785){186}{\line(0,-1){.21639785}}
\multiput(56.25,50.25)(.0672043,-.21639785){186}{\line(0,-1){.21639785}}
\multiput(83.25,50.5)(.067375887,.092789598){423}{\line(0,1){.092789598}}
\put(70.5,89.25){\circle*{2}}
\put(156.25,89.5){\circle*{2}}
\put(174.25,71){\circle*{2}}
\put(180,51){\circle*{2}}
\put(151.25,11.75){\circle*{2}}
\put(135.75,51){\circle*{2}}
\put(103,23.25){\circle*{2}}
\put(110,23.25){\circle*{2}}
\put(117,23.25){\circle*{2}}
\put(163.5,36.25){\circle*{2}}
\put(161.25,33.25){\circle*{2}}
\put(159,30.25){\circle*{2}}
\put(83.25,89.25){\circle*{2}}
\put(95.75,89.25){\circle*{2}}
\put(112.25,89.25){\circle*{2}}
\put(83.25,50){\circle*{2}}
\put(28.5,50){\circle*{2}}
\put(56.25,50){\circle*{2}}
\put(69,10){\circle*{2}}
\put(56.75,10.25){\circle*{2}}
\put(43.75,10){\circle*{2}}
\multiput(136.25,52)(.067434211,.124177632){304}{\line(0,1){.124177632}}
\multiput(136,51.25)(.125,.067434211){304}{\line(1,0){.125}}
\multiput(135.75,51.25)(.06730769,-.17200855){234}{\line(0,-1){.17200855}}
\thinlines
\multiput(179.75,51)(-18.96875,-.0625){8}{\line(-1,0){18.96875}}
\put(15.5,48.75){\circle{25.54}}
\put(55,54.25){$3$}
\put(83.25,46){$4$}
\put(131.5,55){$n$}
\put(159,92){$a^1_n$}
\put(173.5,76){$a^2_n$}
\put(185.75,51){$a^3_n$}
\put(151.25,7.25){$a^n_n$}
\end{picture}
\end{center}
Then $P_h(Z_n)=\{4n-7,4n-6,4n-5,\ldots\}\cup\{+\infty\}$, which leads to the desired result by
Theorem~\ref{salam30}.
\end{proof}
\begin{lemma}\label{salam95}
For all $p\geq0$ with $p\equiv 0(\mod 4)$ there exists graph $X$ with $P_h(X)=\{p,p+1,\ldots\}\cup\{+\infty\}$.
\end{lemma}
\begin{proof}
For $n\geq4$ suppose $W_n$ is the following graph:
\vspace{3mm}
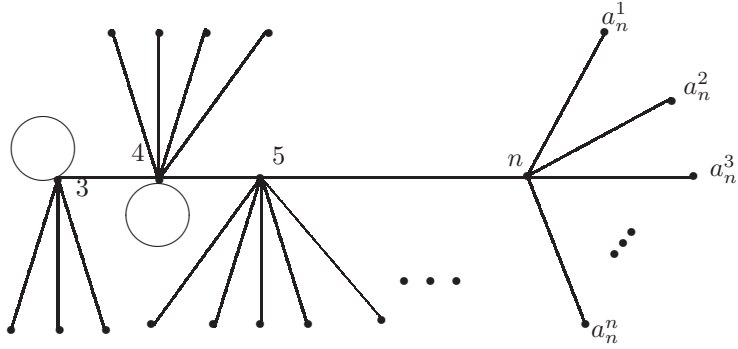
\begin{center}
\unitlength .4mm 
\linethickness{0.4pt}
\ifx\plotpoint\undefined\newsavebox{\plotpoint}\fi 
\begin{picture}(195.5,88.5)(0,0)
\thicklines
\put(49,47.25){\line(0,1){39.75}}
\put(76.25,47.75){\line(0,-1){39.75}}
\put(22.25,47.25){\line(0,-1){39.75}}
\multiput(49,47)(.0672043,.21639785){186}{\line(0,1){.21639785}}
\multiput(76.25,48)(-.0672043,-.21639785){186}{\line(0,-1){.21639785}}
\multiput(49,46.75)(-.0672043,.21639785){186}{\line(0,1){.21639785}}
\multiput(76.25,48.25)(.0672043,-.21639785){186}{\line(0,-1){.21639785}}
\multiput(22.25,47.5)(-.0672043,-.21639785){186}{\line(0,-1){.21639785}}
\multiput(22.25,47.25)(.0672043,-.21639785){186}{\line(0,-1){.21639785}}
\multiput(49.25,47.5)(.067375887,.092789598){423}{\line(0,1){.092789598}}
\multiput(76,47.5)(-.067375887,-.092789598){423}{\line(0,-1){.092789598}}
\put(36.5,86.25){\circle*{2}}
\put(88.75,8.75){\circle*{2}}
\put(167.5,86.5){\circle*{2}}
\put(185.5,68){\circle*{2}}
\put(191.25,48){\circle*{2}}
\put(162.5,8.75){\circle*{2}}
\put(147,48){\circle*{2}}
\put(114.25,20.25){\circle*{2}}
\put(121.25,20.25){\circle*{2}}
\put(128.25,20.25){\circle*{2}}
\put(174.75,33.25){\circle*{2}}
\put(172.5,30.25){\circle*{2}}
\put(170.25,27.25){\circle*{2}}
\put(49.25,86.25){\circle*{2}}
\put(76,8.75){\circle*{2}}
\put(61.75,86.25){\circle*{2}}
\put(63.5,8.75){\circle*{2}}
\put(78.25,86.25){\circle*{2}}
\put(76,47.5){\circle*{2}}
\put(108.25,9.75){\circle*{2}}
\put(47,8.75){\circle*{2}}
\put(49.25,47){\circle*{2}}
\put(22.25,47){\circle*{2}}
\put(35,7){\circle*{2}}
\put(22.75,7.25){\circle*{2}}
\put(9.75,7){\circle*{2}}
\multiput(147.5,49)(.067434211,.124177632){304}{\line(0,1){.124177632}}
\multiput(147.25,48.25)(.125,.067434211){304}{\line(1,0){.125}}
\multiput(147,48.25)(.06730769,-.17200855){234}{\line(0,-1){.17200855}}
\put(141.75,50.75){$n$}
\put(166.75,88.5){$a_n^1$}
\put(188.5,70){$a_n^2$}
\put(195.5,48){$a_n^3$}
\put(164,5.25){$a_n^n$}
\thinlines
\put(18.25,55.5){\circle{18.06}}
\put(48.75,37.75){\circle{18.06}}
\put(191,47.75){\line(-1,0){169}}
\multiput(76,48)(.067468619,-.078974895){478}{\line(0,-1){.078974895}}
\put(79.25,51.25){$5$}
\put(41.75,52){$4$}
\put(27,42.75){$3$}
\end{picture}
\end{center}
Then $P_h(W_n)=\{4n-8,4n-7,4n-6,\ldots\}\cup\{+\infty\}$. Also suppose $W$ is the graph $\bigcirc\!\frac{\:\:\:}{}$, then $P_h(W)=\{4,5,6,\ldots\}\cup\{+\infty\}$, which completes the proof by Theorem~\ref{salam35}.
\end{proof}
\begin{theorem}\label{main}
For all $p\geq0$ there exists graph $\mathcal{G}$ with $P_h(\mathcal{G})=\{p,p+1,\ldots\}\cup\{+\infty\}$.
\end{theorem}
\begin{proof}
Use Lemmas~\ref{salam70}, \ref{salam80}, \ref{salam90} and \ref{salam95}.
\end{proof}
\begin{counterexample}\label{taha10}
The unit interval $[0,1]$ is a subgraph of unit circle $\mathbb{S}^1$ and star graph with three edges
$\ast_3$, however
$P_h([0,1])\nsubseteq P_h(\ast_3)$ and
$P_h(\mathbb{S}^1)\nsubseteq P_h([0,1])$.
\end{counterexample}
\begin{theorem}\label{taha20}
In the graph $X$, the following statements are equivalent:
\begin{itemize}
\item[1.] $X$ is $\mathbb{S}^1$ (i.e., $X$ and ${\mathbb S}^1$ are homeomorphic spaces),
\item[2.] $P_h(X)=\{0,1,\ldots\}\cup\{+\infty\}$,
\item[3.] $0\in P_h(X)$,
\item[4.] all vertices of $X$ have degree $2$.
\end{itemize}
\end{theorem}
\begin{proof}
Using Theorem~\ref{salam35}, we have ``(1) $\Rightarrow$ (2)''. Moreover by \cite[Corollary 9.6]{nadler}, we have ``(4) $\Rightarrow$ (1)''. We should just prove ``(3) $\Rightarrow$ (1)''. Suppose $D$ denotes the collection of all vertices of $X$ with degree else $2$, moreover suppose $D\neq\varnothing$, then for all group $G$ in transformation group $(G,X)$, $D\subset X$ are two invariant closed subsets of $X$, thus $h(G,X)\geq1$ and $0\notin P_h(X)$.
\end{proof}
\begin{theorem}\label{taha30}
For graph $X$ we have $h(Homeo(X),X)<+\infty$ and $P_h(X)=\{h(Homeo(X),X),h(Homeo(X),X)+1,\ldots\}\cup\{+\infty\}$.
\end{theorem}
\begin{proof}
It's evident that for $x\in(a,b)\subset X$ we have $[a,b]\subseteq \overline{Homeo(X)x}$, suppose $X$ is not circle and $V$ is the collection of its vertices, also $E$ is the collection of its edges. Then $\{\overline{Homeo(X)x}:x\in X\}\subseteq\{\bigcup \Lambda:\varnothing\neq\Lambda\subseteq V\}\cup \{\bigcup \Lambda:\varnothing\neq\Lambda\subseteq E\}$, so $h(Homeo(X),X)\leq 2^{|V|}+2^{|E|}-3<+\infty$.
\\
If $X$ is a circle, then the result is clear by Theorem~\ref{salam35}, so $X$ is not a circle and has an edge like $[a_1,b_1]$ such that $a_1$ and $b_1$ are vertices with degrees else than $2$. Choose $w\in(a_1,b_1)$. There exists $n\geq1$ and distinct edges $[a_1,b_1],\ldots,[a_n,b_n]$ such that $\overline{Homeo(X)w}=[a_1,b_1]\cup\cdots\cup[a_n,b_n]$ and $a_i,b_i$ s are vertices with degree else than $2$.
\\
For $i\in\{1,\ldots,n\}$ since $Homeo(X)w\cap [a_i,b_i]\neq\varnothing$ there exists homeomorphism $f:X\to X$ with $f(w)\in[a_i,b_i]$ and in particular (considering the degrees of vertices) $f[a_1,b_1]=[a_i,b_i]$ so $f(a_1),f(b_1)\in\{a_i,b_i\}$. As a matter of fact for all $h\in Homeo(X)$ there exists $j$ with $h(w)\in[a_j,b_j]$, so $h(a_1),h(b_1)\in\{a_j,b_j\}$, hence we have $h(a_1),h(b_1)\in\{a_1,\ldots,a_n,b_1,\ldots,b_n\}$.  Without any loss of generality we may
suppose there exists $h\in Homeo(X)$ with not only $h([a_1,b_1])=[a_i,b_i]$ but also $h(a_1)=a_i$ (may be you need to change the role of $a_i,b_i$) thus $h(b_1)=b_i$. Therefore:
\[\{a_1,\ldots,a_n\}\subseteq Homeo(X)a_1\subseteq\{a_1,\ldots,a_n,b_1,\ldots,b_n\}\tag{A}\]
and
\[\{b_1,\ldots,b_n\}\subseteq Homeo(X)b_1\subseteq\{a_1,\ldots,a_n,b_1,\ldots,b_n\}\:.\tag{B}\]
Note to the fact that for all $x,y\in X$ we have $Homeo(X)x\cap Homeo(X)y\neq\varnothing$ if and only if $Homeo(X)x= Homeo(X)y$, hence by (A) and (B) exactly one of the following cases occurs:
\begin{itemize}
\item $\{a_1,\ldots,a_n\}=Homeo(X)a_1$ and $\{b_1,\ldots,b_n\}= Homeo(X)b_1$,
\item $Homeo(X)a_1=Homeo(X)b_1=\{a_1,\ldots,a_n,b_1,\ldots,b_n\}$.
\end{itemize}
We have the following cases:
\\
{\it Case 1: $Homeo(X)a_1=\{a_1,\ldots,a_n\}$ and $Homeo(X)b_1=\{b_1,\ldots,b_n\}$
(so for all $i\in\{1,\ldots,n\}$, $Homeo(X)a_i=\{a_1,\ldots,a_n\}$ and $Homeo(X)b_i=\{b_1,\ldots,b_n\}$).} Using a similar method described in Theorem~\ref{salam50}, for $m\geq1$
choose:
\[\begin{array}{l} a_1=x_0^1<x^1_1<\cdots<x^1_m=b_1\:, \\ \: \: \: \: \: \vdots \\ a_n=x^n_0<x^n_1<\cdots<x^n_m=b_n\:, \\
x^1_0<\cdots<y^1_{-2}<y^1_{-1}<y^1_0<y^1_1<y^1_2\cdots<x^1_1 \: , \\ \: \: \: \: \: \vdots \\
x^n_0<\cdots<y^n_{-2}<y^n_{-1}<y^n_0<y^n_1<y^n_2\cdots<x^n_1 \: , \end{array}\]
such that
${\displaystyle\lim_{k\to+\infty}y^i_k}=x^i_1$, ${\displaystyle\lim_{k\to+\infty}y^i_{-k}}=x^i_0$
for all $i\in\{1,\ldots,n\}$.
Let $H=\{f\in Homeo(X)\forall i\in\{1,\ldots,n\}\:\exists t\in\{1,\ldots,n\}\:\forall j\in\{0,\ldots,m\}\:(f(\{y^i_k:k\in{\mathbb Z}\})=\{y^t_k:k\in{\mathbb Z}\}$ and $f(x^i_j)=x^t_j)\}$, then
{\small
\begin{eqnarray*}
\{\overline{Hx}:x\in X\} & = & (
\{\overline{Homeo(X)x}:x\in X\}\setminus\{[a_1,b_1]\cup\cdots\cup[a_n,b_n]\})
\\
&& \: \: \: \: \: \cup\{
\{x^1_j,x^2_j,\ldots,x^n_j\}:0<j<m\} \\
&& \: \: \: \: \: \cup\{\{y_k^i:i\in\{1,\ldots,n\},k\in\mathbb{Z}\}\cup\{x^1_0,x^2_0,\ldots,x^n_0,x^1_1,x^2_1,\ldots,x^n_1\}\} \\
&& \: \: \: \: \: \cup\{[x_{j-1}^1,x_j^1]\cup\cdots\cup[x_{j-1}^n,x_j^n]:j\in\{1,\ldots,m\}\}\:.
\end{eqnarray*}}
Thus:
\begin{eqnarray*}
h(H,X) & = & |\{\overline{Hx}:x\in X\} | -1 \\
& = & | \{\overline{Homeo(X)x}:x\in X\}|-1-1+(m-1)+1+m \\
& = & h(Homeo(X),X)+2m-1\:.
\end{eqnarray*}
Now let $K=\{f\in Homeo(X)\forall i\in\{1,\ldots,n\}\:\forall j\in\{0,\ldots,m\}\:f(x^i_j)\in\{x^1_j,\ldots,x^n_j\}\}$, then:
\begin{eqnarray*}
\{\overline{Kx}:x\in X\} & = & (
\{\overline{Homeo(X)x}:x\in X\}\setminus\{[a_1,b_1]\cup\cdots\cup[a_n,b_n]\})
\\
&& \: \: \: \: \: \cup\{
\{x^1_j,x^2_j,\ldots,x^n_j\}:0<j<m\} \\
&& \: \: \: \: \: \cup\{[x_{j-1}^1,x_j^1]\cup\cdots\cup[x_{j-1}^n,x_j^n]:j\in\{1,\ldots,m\}\}\:.
\end{eqnarray*}
Thus:
\begin{eqnarray*}
h(K,X) & = & |\{\overline{Kx}:x\in X\} | -1 \\
& = & | \{\overline{Homeo(X)x}:x\in X\}|-1-1+(m-1)+m \\
& = & h(Homeo(X),X)+2m-2 \:.
\end{eqnarray*}
{\it Case 2: $Homeo(X)a_1=\{a_1,\ldots,a_n,b_1,\ldots,b_n\}$
(so for all $i\in\{1,\ldots,n\}$, $Homeo(X)a_i=Homeo(X)b_i=\{a_1,\ldots,a_n,b_1,\ldots,b_n\}$). }
Using a similar method described in Theorem~\ref{salam30}, for $m\geq1$
choose:
\[\begin{array}{l} a_1=x_0^1<x^1_1<\cdots<x^1_m=b_1\:, \\ \: \: \: \: \: \vdots \\ a_n=x^n_0<x^n_1<\cdots<x^n_m=b_n\:. \end{array}\]
Let $K=\{f\in Homeo(X):\forall i\in\{1,\ldots,n\}\:\forall j\in\{0,\dots,m\} f(x^i_j)\in\{x^1_{m-j},\ldots,x^n_{m-j}\}\}$, then $G=\{f_1\cdots f_s:s\geq1,f_1,\ldots,f_s\in K\}$ is a subgroup of $Homeo(X)$. Moreover:
{\small
\\
$\{\overline{Gx}:x\in X\} $
\begin{eqnarray*}
& = & (
\{\overline{Homeo(X)x}:x\in X\}\setminus\{[a_1,b_1]\cup\cdots\cup[a_n,b_n]\})
\\
&& \cup\{
\{x^1_j,x^2_j,\ldots,x^n_j,x^1_{m-j},\ldots,x^n_{m-j}\}:0<j<m\} \\
&& \cup\{[x_{j-1}^1,x_j^1]\cup\cdots\cup[x_{j-1}^n,x_j^n]\cup
[x_{m-j}^1,x_{m-(j-1)}^1]\cup\cdots\cup[x_{m-j}^n,x_{m-(j-1)}^n]:j\in\{1,\ldots,m\}\}\:.
\end{eqnarray*}}
Thus:
\begin{eqnarray*}
h(G,X) & = & |\{\overline{Gx}:x\in X\} | -1 \\
& = & | \{\overline{Homeo(X)x}:x\in X\}|-1-1+m \\
& = & h(Homeo(X),X)+m-1\:.
\end{eqnarray*}

\end{proof}
\subsection*{More questions}
On possible heights of a topological space there exist the following studies too:
\\
$\bullet$ For Alexandroff square $\mathbb A$ (see \cite{counter}) we have 
	$P_h(\mathbb{A})=\{n:n\geq5\}\cup\{+\infty\}$ \cite{punjab}.
\\
$\bullet$ For unit square $\mathbb{O}=[0,1]\times[0,1]$ equipped with lexicographic order topology we have 
	$P_h(\mathbb{O})=\{n:n\geq4\}\cup\{+\infty\}$ \cite{punjab}.
\\
$\bullet$ For unit square $\mathbb{U}=[0,1]\times[0,1]$ equipped with induced topology of Euclidean plane $\mathbb{R}^2$ 
	we have 
	$P_h(\mathbb{U})=\{n:n\geq1\}\cup\{+\infty\}$ \cite{punjab}.
\\
$\bullet$ For infinite Fort space $F$ (see \cite{counter}) we have
	$P_h(F)=\{n:n\geq1\}\cup\{+\infty\}$ and for finite Fort space $F_0$ with $k$ elements we have
	$P_h(F_0)=\{n:0\leq n\leq k\}$ (use \cite{india}).
\\
In an other point of view in topological space $X$ one may consider a nonempty subclass $\mathcal T$ of subgroups of $Homeo(X)$ and compute $\{h(G,X):X\in{\mathcal T}\}$ \cite{general}. Following this approach, in topological space $X$
compute the following sets:
\begin{itemize}
\item the collection of all $h(G,X)$ such that $G$ is a finitary permutation subgroup of $Homeo(X)$, i.e., $G$ is a subgroup
	of $Homeo(X)$ and for all $g\in G$ the set $\{x\in X: gx\neq x\}$ is finite,
\item the collection of all $h(G,X)$ such that $G$ is a cofinitary permutation subgroup of $Homeo(X)$, i.e., $G$ is a subgroup
	of $Homeo(X)$ and for all $g\in G\setminus\{id_X\}$ the set $\{x\in X: gx=x\}$ is finite.
\end{itemize}

\noindent
{\small  
{\bf Fatemah Ayatollah Zadeh Shirazi},
Faculty of Mathematics, Statistics and Computer Science,
College of Science, University of Tehran,
Enghelab Ave., Tehran, Iran
\linebreak
({\it e-mail}: f.a.z.shirazi@ut.ac.ir  ,  fatemah@khayam.ut.ac.ir)
\vspace{3mm}
\\
{\bf Arezoo Hosseini},
Faculty of Mathematics, College of Science, Farhangian University, Pardis Nasibe--shahid sherafat, Enghelab Ave., Tehran, Iran
({\it e-mail}: a.hosseini@cfu.ac.ir)
\vspace{3mm}
\\
{\bf Zahra Nili Ahmadabadi},
Science and Research Branch, Islamic Azad University, 
Tehran, Iran 
({\it e-mail}: zahra.nili.a@gmail.com)}

\end{document}